\documentclass[a4paper,10pt]{article}
\title{Some considerations on amoeba forcing notions}
\author{Giorgio Laguzzi}

\usepackage{fontenc}
\usepackage{amsmath}
\usepackage{amsfonts}
\usepackage{amssymb}
\usepackage{amsthm}
\usepackage{graphicx}
\usepackage[english]{babel}

\newtheorem{definition2}{Definition}

\newtheorem{lemma}[definition2]{Lemma}
\newtheorem{claim}[definition2]{Claim}
\newtheorem{remark2}[definition2]{Remark}

\newtheorem*{Mobservation2}{Main Observation}
\newtheorem{example2}[definition2]{Example}
\newtheorem{corollary}[definition2]{Corollary}

\newenvironment{definition}{\begin{definition2} \upshape}{\end{definition2}}
\newenvironment{remark}{\begin{remark2} \upshape}{\end{remark2}}

\newcommand{\asacks}{\poset{AS}}
\newcommand{\asilver}{\poset{AV}}

\newcommand{\baire}{\omega^\omega}

\newcommand{\cantor}{2^\omega}
\newcommand{\cantorfin}{2^{< \omega}}
\newcommand{\cohen}{\poset{C}}
\newcommand{\conc}{\smallfrown}

\newcommand{\DDelta}{\mathbf{\Delta}}

\newcommand{\enfa}{\textit}

\newcommand{\ifif}{\Leftrightarrow}
\newcommand{\force}{\Vdash}

\newcommand{\good}{\overset{g}{\ll}}

\newcommand{\height}{\text{ht}}

\newcommand{\levnumber}{\texttt{SL}}

\newcommand{\model}{\textsc}

\newcommand{\poset}{\mathbf}

\newcommand{\random}{\poset{B}}

\newcommand{\restric}{{\upharpoonright}}

\newcommand{\sacks}{\poset{S}}

\newcommand{\silver}{\poset{V}}

\newcommand{\SSigma}{\mathbf{\Sigma}}

\newcommand{\splitting}{\textsc{Split}}
\newcommand{\splitlevel}{\textsf{ns}}
\newcommand{\stem}{\textsc{Stem}}

\newcommand{\term}{\textsc{Term}}

\begin{document}
\maketitle

\begin{abstract}
In this paper we analyse some notions of amoeba for tree forcings. In particular we introduce an amoeba-Silver and prove that it satisfies quasi pure decision but not pure decision. Further we define an amoeba-Sacks and prove that it satisfies the Laver property. We also show some application to regularity properties. We finally present a generalized version of amoeba and discuss some interesting associated questions.
\end{abstract}

\textbf{Acknowledgement} For the first part of the present paper, the author wishes to thank Sy Friedman and the FWF for the indispensable support through the research project \#P22430-N13.
\section{Introduction}
The amoeba forcings play an important role when dealing with questions concerning the real line, such as cardinal invariants and regularity properties. As an intriguing example, one may consider the difference between the amoeba for measure and category in Shelah's proof regarding the use of the inaccessible cardinal to build models for regularity properties, presented in \cite{Sh84} and \cite{Sh85}; in fact, since the amoeba for category is \emph{sweet} (a strengthening of $\sigma$-centeredness), one can construct, via amalgamation, a Boolean algebra as limit of length $\omega_1$ (without any need of the inaccessible), in order to get an extension where all projective sets have the Baire property. On the contrary, for Lebesgue measurability, Shelah proved that if we assume all $\SSigma^1_3$ sets to be Lebesgue measurable, we obtain, for all $x \in \baire$, $\model{L}[x] \models \text{``$\omega_1^\model{V}$ is inaccessible''}$. If one then goes deeply into Shelah's construction of the model satisfying projective Baire property just mentioned, one can realize that the unique difference with Lebesgue measurability consists of the associated amoeba forcing, which is not sweet for measure. Such an example is probably one of the oldest and most significant ones to underline the importance of the amoeba forcing notions in set theory of the real line. In other cases, it is interesting to define amoeba forcings satisfying certain particular features, like not adding specific types of generic reals, not collapsing $\omega_1$ and so on; these kinds of constructions are particularly important when one tries to separate regularity properties of projective sets, or when one tries to blow up certain cardinal invariants without affecting other ones. For a general and detailed approach to regularity properties, one may see \cite{K12}. The main aim of the present paper is precisely to study two versions of amoeba, for Sacks and Silver forcing, respectively. 
\begin{definition} \label{def:amoeba}
Let $\poset{P}$ be either Sacks or Silver forcing. We say that $\poset{AP}$ is an amoeba-$\poset{P}$ iff for any ZFC-model $\model{M} \supseteq \model{N}^{\poset{AP}}$, we have
\[
\model{M} \models \forall T \in \poset{P} \cap \model{N} \quad \exists T' \in \model{M} \cap \poset{P} \quad  (T' \subseteq T \land \forall x \in [T'](\text{$x$ is $\poset{P}$-generic over $\model{N}$})).
\]
\end{definition}

Note that this definition works even when $\poset{P}$ is any other tree forcing
notions (Laver, Miller, Mathias, and so on). We would like to mention that a similar work for Laver and Miller forcing is developed in detail by Spinas in \cite{S95} and \cite{S09}.

Let us now recall some basic notions and standard notation. Given $t,t' \in \cantorfin$, we write $t' \unlhd t$ iff $t'$ is an initial segment of $t$. A \emph{tree} $T$ is a subset of $\cantorfin$ closed under initial segments, i.e., for every $t \in T$, for every $k < |t|$, $t \restric k \in T$,
 where $|t|$ represents the length of $t$.
Given $s,t \in T$, we say that $s$ and $t$ are \emph{incompatible} (and we write $s \nparallel t$) iff neither $s \unlhd t$ nor $t \unlhd s$; otherwise one says that $s$ and $t$ are compatible ($s \parallel t$).
We denote with $\stem(T)$ the longest element $t \in T$ compatible with every node of $T$.
For every $t \in T$, we say that $t$ is a \emph{splitting node} whenever both $t^\conc 0 \in T$ and $t^\conc 1 \in T$, and we denote with $\splitting(T)$ the set of all splitting nodes.
Moreover, for $n \geq 1$, we say $t \in T$ is an \emph{$n$th splitting node} iff   $t \in \splitting(T)$ and there exists $n \in \omega$ maximal such that there are natural numbers $k_0 < \dots < k_{n-1}$ with
$t \restric k_j \in \splitting(T)$,
for every $j \leq n-1$. We denote with $\splitting_n(T)$ the set consisting of the $n$th splitting nodes. For a finite tree $T$, the \emph{height} of $T$ is defined by $\height(T):= \max\{n: \exists t \in T, |t|=n \}$, while $\term(T)$ denotes the set of terminal nodes of $T$, i.e, those nodes having no proper extensions in $T$.  
Finally, for every $t \in T$, the set $\{s \in T: s \parallel t \}$ is denoted by $T_t$, the \emph{body of $T$} is defined as $[T]:= \{ x \in \cantor: \forall n \in \omega (x \restric n \in T) \}$, and $T | n := \{ t \in T: |t| \leq n \}$.

Further, given a tree $T$ and a finite subtree $p \subset T$, we define:
\begin{itemize}
\item $T {\downarrow} p := \{ t \in T: \exists s \in \term(p) (s \parallel t)  \}$;
\item $p \sqsubseteq T \Leftrightarrow \forall t \in T \setminus p \quad \exists s \in \term(p) (s \unlhd t)$, 
and we will say that $p$ \enfa{is an initial segment of $T$}, or equivalently $T$ \enfa{end-extends} $p$.
\end{itemize}

Our attention is particularly focused on the following two types of infinite trees of $\cantorfin$:
\begin{itemize}
\item $T \subseteq \cantorfin$ is  a \enfa{perfect} (or \emph{Sacks}) tree iff each node can be extended to a splitting node.
\item $T \subseteq \cantorfin$ is a \enfa{Silver tree} (or \enfa{uniform tree}) iff $T$ is perfect and for every $s,t \in T$, such that $|s|=|t|$, one has
$s^\conc 0 \in T \ifif t^\conc 0 \in T$ and $s^\conc 1 \in T \ifif t^\conc 1 \in T$.
\end{itemize}
Sacks forcing $\sacks$ is defined as the poset consisting of Sacks trees, ordered by inclusion, and Silver forcing $\silver$ is analogously defined by using Silver trees. Further, if $G$ is the $\sacks$-generic filter over $\model{N}$, we call the generic
branch $z_G= \bigcup\{ \stem(T): T \in G \}$ a Sacks
real (and analogously for Silver).
Other common posets that will appear in the paper will be the Cohen forcing $\cohen$, consisting of finite sequences of $0$'s and $1$'s, ordered by extension,
and the random forcing $\random$, consisting of perfect trees $T$ with strictly positive measure, ordered by inclusion.

We recall the notion of \emph{axiom A}, which is a strengthening of properness.
\begin{definition} \label{def:axiomA}
A forcing $P$ satisfies \enfa{Axiom A}  \index{axiom A} if and only if there exists a sequence $\{ \leq_n: n \in \omega \}$ of orders of $P$ such that:
\begin{enumerate}
\item for every $a,b \in P$, for every $n \in \omega$, $b \leq_{n+1} a$ implies both $b \leq_{n} a$ and $b \leq a$;
\item for every sequence $\langle a_n : n \in \omega \rangle$ of conditions in $P$ such that  for every $n \in \omega, a_{n+1} \leq_n a_n$, there exists $b \in P$ such that for every $n \in \omega$, $b \leq_n a_n$;
\item for every maximal antichain $A \subseteq P$, $b \in P$, $n \in \omega$, there exists $b' \leq_n b$ such that $\{ a \in A: a \text{ is compatible with } b' \}$ is countable.
\end{enumerate}
\end{definition}

\textbf{Notational convention.} In the literature, the Silver forcing is usually denoted by $\silver$, and we keep such a convention. As a consequence, to avoid possible confusion, the ground model will be denoted by the letter \model{N}, instead of the more common $\model{V}$. 

The paper is organized as follows: in section \ref{section:silver},  we show that the natural amoeba-Silver satisfies axiom A, and so in particular does not collapse $\omega_1$. In section \ref{section:sacks}, we introduce a version of amoeba-Sacks and prove that it satisfies the Laver property. We remark that our construction is very different from the one presented by Louveau, Shelah and Velickovic in \cite{LSV}, and in particular we do not use any strong partition theorem (like Halpern-La\"uchli theorem). Finally, a last section is devoted to discuss some difficulties when trying to kill Cohen reals added by the amoeba-Silver, and we discuss a generalized notion of amoeba together with some possible further developments concerning regularity properties. At the end of section \ref{section:sacks}, we also present an application of amoeba-Sacks, to separate Sacks measurability from Baire property at some projective level.

\

\textbf{Acknowledgement.}
I would like to thank Yurii Khomskii and Wolfgang Wohofsky for some stimulating and helpful discussions concerning amoeba for tree-forcings. Further, I thank Martin Goldstern to suggest me Spinas' papers \cite{S95} and \cite{S09}, which have been crucial for a deep understanding of this topic, as well as some enlightening insights.

\section{Amoeba-Silver} \label{section:silver}

In this section we discuss some properties of the amoeba-Silver $\asilver$ defined by: 
\[
(p,T) \in \asilver \text{ iff } T \text{ is a Silver tree and } \exists n \in \omega \text{ such that } p= T | n,
\]
ordered by $(p',T') \leq (p,T) \Leftrightarrow T'\subseteq T \land p' \restric \height(p)=p$. 

For a proof that this is a well-defined notion of amoeba for Silver forcing, i.e, it satisfies Definition \ref{def:amoeba}, one can see \cite[Lemma 18, Corollary 20]{L14}.

In order to show that $\asilver$ satisfies axiom A, we define the sequence of orders on $\asilver$ as follows:
\begin{eqnarray*}
(p',T') \leq_n (p,T) &\ifif& (p',T') \leq (p,T) \\
						 &&   p'=p \wedge \forall k \leq n (\splitting_{k}(T')= \splitting_{k}(T)).
\end{eqnarray*}

Clearly, conditions 1 and 2 of Definition \ref{def:axiomA} are satisfied. To obtain condition 3, we first check Lemma \ref{fact:VT-axiomA}, i.e., $\asilver$ satisfies quasi pure decision.

\begin{definition}
Given $D \subseteq \asilver$ open dense, $(p,T) \in \asilver$ and $q$ finite subtree of $T$, we say that $q$ is \enfa{deciding} iff there exists $S \subseteq T$ such that $(q,S) \in D$.

\end{definition}
\begin{lemma} \label{fact:VT-axiomA}
Let $D \subseteq \asilver$ be open dense, $(p,T) \in \asilver$ and $m \in \omega$. Then there exists $T^* \subseteq T$ such that $(p,T^*) \leq_m (p,T)$ and 
\[
 q \text{ is deciding } \Rightarrow  (q, T^* {\downarrow} q) \in D.
\]
\end{lemma}
\begin{proof}
For every tree $T$, let $\levnumber_T(n):= |t|$, where $t \in T$ is an $n$th splitting node.
Let $D \subseteq \asilver$ be open dense and fix $(p,T) \in \asilver$ arbitrarily. Let $p^0= T | (\texttt{SL}_T(h_0)+1)$,
where the $h_0$-th splitting nodes are the first splitting nodes occurring above $p$, i.e., if $t \in \splitting_{h_0}(T)$, then $|t| > \height(p)$ and there are no splitting nodes $t'$ such that $\height(p) < |t'| < |t|$. We assume $m=h_0$ and leave the general case to the reader.

We use the following notation: given $T$ infinite tree and $p$ finite tree, put
\begin{multline*}
T \otimes p  := \{ t \in \cantorfin: \exists t' \in T \;\exists t'' \in \term(p)\text{ s.t. }\forall n <  |t''| (t(n)=t''(n)) \\ 
\wedge\; \forall n \geq |t''| (t(n)=t'(n))  \}. \index{$\otimes$}
\end{multline*}

(Intuitively, $T \otimes p$ is the translation of $T$ over $p$).

Let $\{ p^0_j: j < 3  \}$ enumerate the uniform finite trees  such that $p \subseteq p^0_j \subseteq p^0$, $\height(p^0_j)=\height(p^0)$
and $p^0_j {\restric} \height(p)=p$.
Starting from such $p^0$, one develops the following construction for $i \geq h_0$ and $j < 3^{i-h_0+1}$.
\begin{itemize}
\item Start from $i=h_0$:
\begin{itemize}
\item \textsc{Substep} $j=0$: if there exists $S \subseteq T$ such that $(p^0_0,S) \in D$, then put $T^0_0= S$; otherwise put $T^0_0 = T$;
\item \textsc{Substep }$j+1$: if there exists $S\subseteq T^0_j \otimes p^0_{j+1}$ such that  $(p^0_{j+1},S) \in D$,
then put $T^0_{j+1}=S$; otherwise let $T^0_{j+1}= T^0_j$;
\item when the operation is done for every $j < 3$, put $T^1_* = T^0_{2} \otimes p^0$ and
$p^1= T^1_* | (\texttt{SL}_{T^1_*}(h_0+1)+1)$; furthermore, let $\{ p^1_j: j < 3^2 \}$ be the enumeration
of all the uniform finite trees such that  $p^1_j \subseteq p^1 $, $\height(p^1_j)=\height(p^1)$ and $p^1_j {\restric} \height(p)=p$;
\end{itemize}

\item \textsc{Step }$i=h_0+k$:
\begin{itemize}
 \item \textsc{Substep} $j=0$: if there exists $S \subseteq T^{k}_*$ such that  $(p^{k}_0,S) \in D$, then put $T^{k}_0=S$; otherwise let $T^{k}_0= T^{k}_*$;
 \item \textsc{Substep }$j+1$: if there exists $S\subseteq T^{k}_j \otimes p^{k}_{j+1}$ such that  $(p^{k}_{j+1},S) \in D$,
then put $T^{k}_{j+1}=S$; otherwise let $T^{k}_{j+1}= T^{k}_j$;
 \item when the operation is done for every $j < 3^{k+1}$, put $T^{k+1}_* = T^{k}_{3^{k+1}-1} \otimes p^{k}$
and $p^{k+1}= T^{k+1}_* | (\texttt{SL}_{T^{k+1}_*}(i+1)+1)$; furthermore, let $\{ p^{k+1}_j: j < 3^{k+2} \}$
be the enumeration of all the uniform finite trees such that  $p^{k+1}_j \subseteq p^{k+1} $, $\height(p^{k+1}_j)=\height(p^{k+1})$ and
$p^{k+1}_j {\restric} \height(p)=p$.
\end{itemize}
\end{itemize}
Once such a construction is finished, one obtains a sequence $\langle T^k_*: k \in \omega \rangle $ such that
$T^{k+1}_* \leq_{h_0 + k} T^{k}_*$ (where $T^0_*=T$). Hence, the tree $T^*$ obtained by fusion, i.e.,
$T^*= \bigcap_{k \in \omega} T^k_*$, is a Silver tree, and so
the pair $(p,T^*)$ belongs to $\asilver$ and $(p,T^*) \leq_{h_0} (p,T)$.

By construction, one gets
\[ \label{equation:VT-axiomA}
\forall (q,S) \leq (p,T^*), \text{ if } (q,S) \in D \text{ then } (q, T^* {\downarrow} q) \in D,
\]
which completes the proof.
\end{proof}

Lemma \ref{fact:VT-axiomA} is the core of the next key result.
\begin{lemma} \label{lemma:silver-3cond-axA}
Let $A \subseteq P$ be a maximal antichain, $(p,T) \in \asilver$ and $m \in \omega$. Then there exists $T^* \subseteq T$ such that $(p,T^*) \leq_m (p,T)$ and 
$(p,T^*)$ has only countably many compatible elements in $A$.
\end{lemma}
\begin{proof}
Fix a condition $(p,T) \in \asilver$.
Let $D_A$ be the open dense subset associated with $A$, i.e.,
$D_A = \{ (q,S) \in \asilver: \exists (q',S') \in A ((q,S)\leq (q',S'))  \}.$
Let $T^*$ be as in Lemma \ref{fact:VT-axiomA}.
To reach a contradiction, assume there are uncountably many elements in $A$ compatible with $(p,T^*)$, i.e., there is a sequence
$\langle (p_\alpha, T_\alpha) : \alpha < \omega_1 \rangle$ of distinct elements of $A$ and there are $(q_\alpha, S_\alpha)$'s such that,
for every $\alpha < \omega_1$,
\[
(q_\alpha,S_\alpha) \leq (p_\alpha,T_\alpha), (p,T^*).
\]
Note that $(q_\alpha,S_\alpha) \in D_A$. Thus, by Lemma \ref{fact:VT-axiomA}, one obtains $(q_\alpha, T^* {\downarrow} q_\alpha) \in D_A$,
and therefore
\[
(q_\alpha,T^* {\downarrow} q_\alpha)  \leq (p_\alpha,T_\alpha), (p,T^*).
\]
Note that there are only countably many different $q_\alpha$'s and therefore there exist $\alpha_0,\alpha_1 < \omega_1$
such that $(q_{\alpha_0},T^* {\downarrow} q_{\alpha_0})=(q_{\alpha_1},T^* {\downarrow} q_{\alpha_1})$, and this contradicts
$(p_{\alpha_0},T_{\alpha_0}) \perp (p_{\alpha_1},T_{\alpha_1})$.
\end{proof}

\begin{corollary}
$\asilver$ satisfies axiom A. 
\end{corollary}
\begin{proof}
Conditions 1 and 2 of Definition \ref{def:axiomA} are straightforward, while condition 3 follows from Lemma \ref{lemma:silver-3cond-axA}.
\end{proof}

\begin{remark}
 Consider the following notation:
\begin{itemize}
\item for every $p \subset \cantorfin$ finite and uniform, let $\splitlevel(p)=$ number of splitting levels of $p$;
\item let $\Delta_p= \langle \levnumber_p(0), \levnumber_p(1), \dots, \levnumber_p(\splitlevel(p)-1) \rangle$.
\end{itemize}
Finally, if $G$ is $\asilver$-generic over $\model{N}$, let $h= \bigcup \{ \Delta_p: (p,T) \in G \}$.

\textsc{Claim}: $\force_{\asilver} \text{`` }\dot{h} \text{ is dominating over $\model{N}$"}.$

\begin{proof} Fix an increasing $x \in \baire \cap \model{N}$ and $(p,T) \in \asilver$, arbitrarily.
Pick $T' \subseteq T$, $T' \restric \height(p)= p$ such that for every
$n \geq \splitlevel(p)$, $\levnumber_{T'}(n) > x(n) $. It is clear that $(p,T') \force \forall n \geq \splitlevel(p)(\dot h(n) > x(n))$.
\end{proof}
\end{remark}

Amoeba-Silver does not have pure decision, as pointed out by the following observation.
\begin{remark} \label{remark:silver-cohen}
Let $T_G$ be the generic tree added by $\asilver$ and define the following $c \in \cantor$: for every $n \in \omega$,
\[
c(n)=\begin{cases} 0 & \mbox{if } \{ j \in \omega: \levnumber_{T_G}(n+1) + 2 < j \leq \levnumber_{T_G}(n+2)+1 \land T_G(j)=0 \}
\\ & \mbox{ is even;} \\ 1 & \mbox{otherwise. }
\end{cases}
\]
(Intuitively, $c(n)$ represents the parity of $0$s between the $n+1$st and $n+2$nd splitting level.)

\textsc{Claim}: $\force_{\asilver} \text{`` }\dot{c} \text{ is Cohen over $\model{N}$"}.$

\begin{proof} Fix a closed nowhere dense set $F$ of the ground model. Given $(p,T) \in \asilver$, let $c_0$ be the
part of $c$ already decided by such a condition. Denote with $s$ the sequence in $\cantorfin$ such that $[{c_0}^\conc s] \cap F = \emptyset$. Now, it is clear that one can remove some splitting nodes and choose 0 if needed, according to what $s$ tells us; more precisely, for every $n$, $|c_0|\leq n < |{c_0}^\conc s|$, if ${c_0}^\conc s(n)=0$ and there is an even number of $0$s between the $n+1$st and the $n+2$nd splitting level then we do nothing, otherwise, we remove the $n+2$nd splitting level, and we freely choose $0$ or $1$ in order to have an even number of $0$s between the $n+1$st and the ``new'' $n+2$nd splitting level. An analogous argument works when ${c_0}^\conc s(n)=1$.
\end{proof}
\end{remark}

\section{Amoeba-Sacks} \label{section:sacks}

The standard amoeba-Sacks consists of the set of pairs $(p,T)$, where $T$ is a perfect tree and $p=T {\mid} n$, for some $n \in \omega$, ordered by $(p',T') \leq (p,T)$ iff $T' \subseteq T$ and $p'$ end-extends $p$. However such a forcing has the bad feature of adding a Cohen real: let $T_G= \bigcup \{ p: \exists T ((p,T) \in G) \}$, where $G$ is the generic over $\model{N}$, we define, for every $n \in \omega$,
\begin{itemize}
\item[$c(n)= 0$] iff the shortest $n+2$nd splitnode above the leftmost $n+1$st 
 splitnode $t \in T_G$ extends $t^\conc 0$, or if the two $n+2$nd splitting  nodes extending $t$ have the same length; 
\item[$c(n)=1$] otherwise.
\end{itemize}

\begin{claim}
 $c$ is Cohen over the ground model $\model{N}$. 
\end{claim}
\begin{proof} Let $B \in \model{N}$ be closed nowhere dense and $(p,T)$ an amoeba condition. We aim at finding a condition $(p',T') \leq (p,T)$ such that $(p',T') \force \dot{c} \notin B$. Let $t_0 \in \cantorfin$ such that $(p,T) \force t_0 \lhd \dot{c}$, and pick $s_0 \in \cantorfin$ such that $[{t_0}^\conc s_0]  \cap B = \emptyset$. We can then extend $p$ to $p'$ in order to \emph{follow} $s_0$, since we can freely choose the subsequent splitting nodes extending the leftmost branch. Hence, $(p',T') \force {t_0}^\conc s_0 \lhd \dot{c} \notin B$.   

\end{proof}

We are therefore interested in introducing a finer version of amoeba-Sacks which does not add Cohen reals. Actually we will do more, by showing that our forcing satisfies the Laver property.

Before going on we need to introduce some notation:
\begin{itemize}
\item given a perfect tree $T$, consider the natural $\unlhd$-isomorphism $e: \splitting(T) \to 2^{<\omega}$ and put on $\splitting(T)$ the following order:  
\[
s \preccurlyeq t \text{ iff } |e(s)| < |e(t)| \vee ( |e(s)|=|e(t)| \land e(s) \leq_{\text{lex}} e(t)). 
\]
\[
s \prec t \text{ iff } s \preccurlyeq t \land s \neq t.
\]
We will say that $t \in \splitting(T)$ has \emph{depth} $n$ (and we will write $d(t,T)=n$)  iff there is a maximal $n \in \omega$ such that there are $t_0, \dots, t_{n-1} \in \splitting(T)$ with $t_0 \prec \dots t_{n-1} \prec t$ (in case there are no such $t_j$'s we say that $t$ has depth $0$, i.e., $t$ is the $\stem(T)$.) 
\item $T \restric^* n := \{ t \in T: \exists k \leq n \exists s \in \splitting(T) \exists i \in \{ 0,1 \} (s \text{ has depth $k$} \land (t \unlhd s^\conc i) \}$. 
\item Given $T,T'$ perfect trees, we define 
\[  
T' \subseteq_n T \Leftrightarrow T' \subseteq T \land T' \restric^* n = T\restric^* n.
\]
\end{itemize}
\begin{definition}
We say that a tree $T$ is \enfa{good} iff for every $s,t \in \splitting(T)$, one has 
$s \preccurlyeq t \Rightarrow |s| \leq |t|$. 
\end{definition}
We then define our version of amoeba-Sacks $\asacks$ as follows: a pair $(p,T) \in \asacks$ iff $T$ is a good perfect tree  
and $p \sqsubset T$.
The order is given by $(p',T') \leq (p,T)$ iff $T' \subseteq T$ and $p'$ end-extends $p$.

\begin{remark}
Given a perfect tree $T$ there exists a good perfect tree $T' \subseteq T$. In fact, we can build a sequence $\{T_n : n \in \omega  \}$ such that for every $n \in \omega$, $T_{n+1} \subseteq_n T_n$ and $T_n \restric^* n$ is good, by using the following recursive pruning-argument: 
\begin{itemize}
\item start from $T_0:=T$;
\item assume $T_n$ already defined and pick the node $t \in T_n$ with $d(t, T_n)=n$. If $T_n {\restric}^* n$ is good, then put $T_{n+1}:=T_n$; otherwise, \emph{cut} the splitting node, by removing the part of $T_n$ above $t^\conc 1$, go to the next splitting node and keep cutting as far as one finds a tree $S$ so that $S {\restric}^* n$ be good. Let $T_{n+1}:=S$.
\item Put $T':= \bigcap_{n \in \omega} T_n$.
\end{itemize}

Throughout this section, we will use the symbol $T' \good T$ for denoting the good perfect subtree $T'$ of $T$, built via this pruning-argument. Note that such $T'$ is uniquely determined.
\end{remark}

First of all, we check that the name amoeba-Sacks be justified.
\begin{lemma}
Let $G$ be $\asacks$-generic over $\model{N}$ and let $\model{M} \supseteq \model{N}[G]$ be a ZFC-model. Then 
\[
\model{M} \models \forall T \in \model{N}  \cap \sacks \quad \exists T' \in \model{M} \cap \sacks \quad ([T'] \subseteq [T] \land [T'] \subseteq \sacks(\model{N})),
\] 
where $\sacks(\model{N})$ is the set of Sacks generic reals over $\model{N}$.
\end{lemma}

\begin{proof} It is analogous to the argument used for other notions of amoeba, see \cite{L14} for Silver and \cite{S95} for Laver and Miller. Since such an argument is not widely known, we give it here for completeness.  We first check that $T_G:= \bigcup \{p: \exists T ((p,T) \in G)  \} \subseteq \sacks(\model{N})$ in $\model{M}$, i.e., every $x \in [T_G] \cap \model{M}$ is Sacks generic over $\model{N}$, and we then see how to find a \emph{copy} of $T_G$ inside any perfect tree $T \in \model{N}$.

Given $(p,T) \in \asacks$ and $D \subseteq \sacks$ open dense, we build $T^* \subseteq T$ as follows: let $\{t_0, \dots, t_n \}$ enumerate all terminal nodes of $p$, and, for every $j \leq n$, pick $T_j \subseteq T_{t_j}$ such that $T_j \in D$; then put $T^* \good \bigcup \{ T_j: j \leq n \}$.
By construction, we obtain $(p,T^*) \force \forall z \in [T_G](H_z \cap D \neq \emptyset)$, where $H_z$ is defined by $H_z= \{ S \in \sacks \cap \model{N}: z \in [S] \}$.

We have just shown $H_z$ to be generic. It is then left to show that it is a filter: towards contradiction, assume there are $T_1, T_2 \in H_z$ incompatible (note that by absoluteness they are incompatible in $\model{N}$ as well). Hence, $[T_1] \cap [T_2]$ is countable, i.e., $[T_1] \cap [T_2]= \{ x_i: i \leq \omega \}$. Then $E:= \{ T \in \sacks: \forall i \in \omega (x_i \notin [T])  \}$ is an open dense set in the ground model \model{N}, and so, by genericity, there is $T \in E$ and $z \in [T]$, which is a contradiction.

We remark that the argument we used so far works not only for $z \in \model{N}[G]$, but even for all $z \in \model{M}$. In fact, the above argument shows that we can find a front $F \subseteq T_G$, i.e., a set such that for every $t \in F$ we have $(T_G)_t \in D$, and so, since being a front is $\Pi^1_1$, it follows that $F$ remains a front in any ZFC-model $\model{M} \supseteq \model{N}[G]$, and so for every $z \in [T_G] \cap \model{M}$, $\model{M} \models H_z \cap D \neq \emptyset$. 

It is then left to show that we can find a tree $T'$ only consisting of Sacks generic reals, inside any perfect $T$ of the ground model. To this aim, it is enough to note that, for any $T \in \sacks \cap \model{N}$,  the forcing $\asacks_{T}$ defined as $\asacks_{T} := \{ (p,S) \in \asacks : S \subseteq T \}$, with the analogous order, is isomorphic to $\asacks$.

\end{proof}

\begin{remark} \label{remark:pmaximal}
Let $(p_0,T) \in \asacks$. By goodness, there exists $p \sqsupseteq p_0$ maximal (w.r.t. $\sqsubseteq $) such that for every $T' \subseteq T$ with $(p_0,T') \in \asacks$ one has $p \sqsubset T'$ (in particular, every $(q,S) \leq (p_0,T)$ is compatible with $(p,T)$ and so the two conditions are forcing equivalent). Note that such $p$ is of the form $T{\restric}^*n$, for some $n$, \emph{but} with every terminal node of the latter extended to the corresponding subsequent splitting node.
\end{remark}

To show that $\asacks$ satisfies the Laver property, we first have to introduce a notion of $\leq_n$: 
\[
(p_0',T') \leq_n (p_0,T) \Leftrightarrow (p_0',T') \leq (p_0,T) \land p_0=p_0' \land T' \subseteq_{n+N} T,
\]
where $N:= \max \{ k \in \omega: \exists t \in \splitting(p)(d(t,p)=k) \}$, with $p \sqsupseteq  p_0$ as in Remark \ref{remark:pmaximal}. 
$\asacks$ satisfies axiom A, and the proof works similarly to the one for amoeba-Silver $\asilver$ viewed in the previous section. In fact, $\asacks$ satifies quasi pure decision, together with an akin version of Lemma \ref{lemma:silver-3cond-axA}.  
\begin{lemma} \label{lemma:well-known}
$\asacks$ has quasi pure decision, i.e., given $D \subseteq \asacks$ open dense, $(p,T) \in \asacks$ and $m \in \omega$, there exists $T^+ \subseteq T$ such that $(p,T^+) \leq_m (p,T)$ and
\[
\text{$q$ is deciding} \Rightarrow (q, T^+ {\downarrow} q) \in D.
\]
\end{lemma}
\begin{proof}[Sketch of the proof.] It is analogous to that of Lemma \ref{fact:VT-axiomA} for $\asilver$. Given $D \subseteq \asacks$ open dense and $(p,T) \in \asacks$, we can build $T^+ \subseteq_m T$ with the desired property, for some arbitrary fixed $m \in \omega$, by using the following inductive argument: start with $q^0=T \restric^* m$ and $T^0=T$. for $j >0 $, let $q_j= T^{j-1} \restric^* (m+j)$. Then use an analogous \emph{shrinking-argument} as in the proof of Lemma \ref{fact:VT-axiomA} in order to get $T^{j} \subseteq_{m+j} T^{j-1}$ so that 
\[
\forall q (p \sqsubseteq q \subseteq q_j \land q \text{ is deciding } \Rightarrow (q, T^j {\downarrow} q) \in D).
\]    
Finally put $T^+= \bigcap_{j \in \omega} T_j$. We then get $(p,T^+) \leq_{m} (p,T)$ with the required property.
\end{proof}

Note that even the standard amoeba-Sacks satisfies quasi pure decision, and the argument for proving that is analogous.
\begin{lemma}
$\asacks$ has pure decision, i.e., given a formula $\varphi$ and a condition $(p_0,T) \in \asacks$, there exists $(p_0,T') \leq_0 (p_0,T)$ such that $(p_0,T') \in D$, with $D =\{ (q,S) \in \asacks: (q,S) \force \varphi \vee (q,S) \force \neg \varphi \}$.   
\end{lemma}
\begin{proof}
First of all, let $p \sqsupseteq p_0$ be as in Remark \ref{remark:pmaximal}.
The idea of the proof by contradiction is the following. Assume there is no $T' \subseteq T$ such that $(p,T') \in D$, and so also no $(p_0,T') \in D$. We will construct $T^* \subseteq T$ such that $(p, T^*) \in \asacks$ and for every $(q,S) \leq (p_0,T^*)$ one has:
\begin{itemize}
\item[$(\star_1)$] if $q$ is deciding then $(q,T^* {\downarrow}q) \in D$ (this can be done by virtue of Lemma \ref{lemma:well-known});
\item[$(\star_2)$] there exists $q'$ such that $q \sqsubseteq q' \sqsubset S$, $q' \sqsupseteq p$ and $(q',T^*{\downarrow}q') \notin D$.
\end{itemize}
This two facts obviously contradict $D$ being dense.

We use the following notation: for every $s \in \splitting(T)$, $p$ finite tree,
\[
p \oplus s :=  \{ t : t \in p \vee \exists i \in \{ 0,1 \} (t \unlhd s^\conc i) \}.
\]

Let $t_0 \in p$ be such that $t_0 = r^\conc 0$, with $r \in \splitting(p)$ satisfying:
\begin{itemize}
\item[(i)] there is no $v \rhd r$ such that $v \in \splitting(p)$, and
\item[(ii)] $r$ has smallest depth with property (i), i.e., for every $u \in \splitting(p)$, if $d(u,p) < d(r,p)$ then there exists $u' \rhd u$ such that $u' \in \splitting(p)$. 
\end{itemize}
(In case $\splitting(p)= \emptyset$, let $t_0= \stem(T)^\conc 0$.)

We can assume $T$ to be as the $T^+$ of Lemma \ref{lemma:well-known}, so that $(\star_1)$ be satisfied. 
The rest of the proof is devoted to building $T^* \subseteq T$ satisfying $(\star_2)$ as well. We split it into three claims. 
\begin{claim} \label{claim1}
There are perfectly many $s_j$'s in $\splitting (T)$ extending $t_0$ such that $p \oplus s_j$ is not deciding.
\end{claim}

\begin{proof}
Assume, towards contradiction, that such a set were not perfect. Then one could find a perfect $P$ consisting of all $t\unrhd t_0$ in $\splitting(T)$ such that $(p \oplus  t , T {\downarrow} (p \oplus  t )) \in D$ and moreover 
\begin{itemize}
\item[(i)] either for all $t \in P$, $(p \oplus  t , T {\downarrow} (p \oplus  t )) \force \varphi$,
\item[(ii)] or for all $t \in P$, $(p \oplus  t , T {\downarrow} (p \oplus  t )) \force \neg \varphi$,
\end{itemize}
Hence, by letting $T^- \good \bigcup_{t \in P} (p \oplus t) \cup \bigcup \{ T_r: r \in T \land r \nparallel t_0 \}$ we would have 
\begin{itemize}
\item[] $\text{(i)} \Rightarrow \forall (q,S) \leq (p,T^-) \exists (q',S') \leq (q,S) ((q',S') \force \varphi) \Rightarrow (p,T^-) \force \varphi$ 
\item[] $\text{(ii)} \Rightarrow \forall (q,S) \leq (p,T^-) \exists (q',S') \leq (q,S) ((q',S') \force \neg \varphi) \Rightarrow (p,T^-) \force \neg \varphi$,
\end{itemize}
and so in both cases $(p,T^-) \in D$, contradicting our initial assumption.
\end{proof}

Let $S^1 := T^-$. Furthermore, note that $(p,S^1) \leq_0 (p,T)$. 
 \begin{claim} \label{claim2} 
Let $t_1 \in \splitting(S^1)$ such that $t_1 = r^\conc 1$, where $r$ is the same as in the definition of $t_0$ above. There exists $W \subseteq S^1_{t_1}$ perfect and good such that for every $u \in \splitting (S^1)$ extending $t_0$, for every $s \in \splitting(W)$, we have $p \oplus u \oplus s$ is not deciding.
\end{claim}
\begin{proof}
Let $u$ be the first splitting node of $S^1$ extending $t_0$. By an analogous argument as in the above lemma, we find perfectly many $s_j \in \splitting(S^1)$ extending $t_1$ such that, for every $j \in \omega$, $p \oplus u \oplus s_j$ is not deciding,  
otherwise $p \oplus u$ would be deciding, contradicting Claim \ref{claim1}. 
Let $R^0:=\{ s_j: j \in \omega \}$, $S^1_0 \good \bigcup_{s \in R^0} (p \oplus s) \cup \bigcup \{ (S^1_0)_t: t \in S^1_0 \land t  \nparallel t_1 \}$ and let $w$ be the shortest node in $S^1_0$ extending $t_1$. 

Then let $u_0$ be the first splitting node of $S^1_0$ extending $u^\conc 0$ and analogously $u_1$ the one extending $u^\conc 1$. By the usual argument, we find perfectly many $s$'s in $\splitting (S^1_0)$ extending $w^\conc 0$ such that
\begin{equation} \label{eq1}
p \oplus u_0 \oplus s \text{ is not deciding}.  
\end{equation} 
Let $P^0_0 \subseteq R^0$ be the set of such perfectly many nodes. Moreover, we also find perfectly many $s \in P^0_0$ such that 
\begin{equation} \label{eq2}
p \oplus u_1 \oplus s \text{ is not deciding}.  
\end{equation}
Let $P^0_1 \subseteq P^0_0$ be the set of such nodes.

A specular argument can be done also for $w^\conc 1$ in order to find $P^1_1 \subseteq R^0$ such that every $s \in P^1_1$ extends $w^\conc 1$ and satisfies both (\ref{eq1}) and (\ref{eq2}). Finally put $R^1= \{ w \} \cup P^0_1 \cup P^1_1$ (note that $R^1$ and $R^0$ have the same first node, namely $w$). 
Then put
\[
\begin{split}
S^1_1 \good  & \bigcup \{ p \oplus u \oplus s: u \in \splitting((S^1_0)_{t_0}), s \in R^1  \}  \cup \\ 
& \bigcup \{ (S^1_0)_t: t \in S^1_0 \land t \nparallel u \land t \nparallel w \}. 
\end{split}
\]
Furthermore let, for $i,j,k \in \{ 0,1 \}$,
\begin{itemize}
\item $w_i \unrhd w^\conc i$ be the first splitting node occurring in $R^1$;
\item $u_{kj} \unrhd {u_k}^\conc j$ be the first splitting node occurring in $S^1_1$ (note that, by goodness, for each $i \in \{ 0,1 \}$, one has $|u_{kj}| > |w_i|$). 
\end{itemize}
Note that $u_0, u_1 \in \splitting(S^1_1)$, since $|u_0|,|u_1| < |w_0|$ by goodness.

By repeating this procedure, we obtain, for $n \in \omega$, $R^n \subseteq R^{n-1}$ such that for  every $s \in R^n$, for every $\sigma \in 2^{\leq n}$,
$p \oplus u_\sigma \oplus s \text{ is not deciding}$, where we identify $u$ with $u_\emptyset$. Moreover, put
\[ 
\begin{split} 
S^1_n  \good & \bigcup \{ p \oplus u \oplus s: u \in \splitting((S^1_{n-1})_{t_0}), s \in R^n  \} \cup \\
& \bigcup \{ (S^1_{n-1})_t: t \in S^1_{n-1} \land t \nparallel u \land t \nparallel w \}.
\end{split}
\]
Note, for every $\sigma \in 2^n$, we have $u_\sigma \in \splitting(S^1_n)$.
Finally, put $R = \bigcap_{n \in \omega} R^n$ and $W= \bigcup \{ t: \exists s \in R (t \unlhd s)  \}$. Note that the definition of $R$ makes sense, since for every $n \in \omega$, $R^{n+1}\cap R^n \supseteq \{ w_\sigma: \sigma \in 2^{\leq n}  \}$, and so the construction is obtained by a kind of standard fusion argument (note that we identify $w$ with $w_\emptyset$). By construction, such $W$ has the required properties.
\end{proof}

Then define $S^2:= \bigcap_{n \in \omega} S^1_n$.
Note that $u^\conc 0, u^\conc 1 \in S^2 \cap S^1$ and therefore $(p,S^2) \leq_1 (p,S^1)$. 
\begin{claim} \label{claim3}
Let $t_n$ be as follows: if $t_{n-1}$ was of the form $r^\conc 0$, then let $t_n = r^\conc 1$; if $t_{n-1}$ was of the form $r^\conc 1$, then let $t_n= z^\conc 0$, where $z$ is the splitting node of $S^{n-1}$ such that $d(z,S^{n-1})=d(r,S^{n-1})+1$. 
There exists $W \subseteq (S^{n-1})_{t_n}$ perfect and good such that for every $A:= ( s_0, \dots, s_{n-1} ) \in (\splitting(S^{n-1}))^{n}$, for every $w \in \splitting(W)$, we have $p(A,w)$ is not deciding, where $p(A,w):= p \oplus s_0 \oplus \dots \oplus s_{n-1} \oplus w$.
\end{claim}

\begin{proof} 
The proof of Claim \ref{claim3} is a generalization of the one of Claim \ref{claim2}. 

Use the following notation: for $w \in \splitting(S^{n-1})$, let 
\[ 
\begin{split}
\mathfrak{A}(w, S^{n-1})= &\{ (s_0, \dots, s_{n-1}) \in (\splitting(S^{n-1}))^{n}: \\
						 &  p \oplus s_0 \oplus \dots \oplus s_{n-1} \oplus w \text{ is good}\}.
\end{split} 
\]
Note that $\mathfrak{A}(w, S^{n-1})$ is always finite. For any $A \in \mathfrak{A}(w, S^{n-1})$, say $A=(s_0, \dots  s_{n-1})$, we will use the notation $p(A,w)=p \oplus s_0 \oplus \dots \oplus s_{n-1} \oplus w$ (for $w \in \splitting(S^{n-1})$). 

We define the set $S^n$ as the limit of the following inductive construction:

\begin{itemize}
\item[\textsc{Step $0$}]: Let $p^+= p \oplus u_0 \oplus \dots \oplus u_{n-1}$, where each $u_j$ is the first splitting node occurring in $S^{n-1}$ extending $t_{j}$. By the usual argument, one can find perfectly many $s_j$'s extending $t_n$ such that, $p^+ \oplus s_j$ is not deciding, otherwise $p^+$ would be deciding.
Let $P_\emptyset$ be the set of such perfectly many $s_j$'s and $w_\emptyset$ its least element. Moreover, let
\[
\begin{split}
S^{n-1}_0  \good  & \bigcup \{ (S^{n-1})_{u_j}: j < n \} \cup \bigcup \{ p^+ \oplus s: s  \in P_{\emptyset} \} \cup  \\
&  \bigcup \{ (S^{n-1})_t: t \in S^{n-1} \land \forall j < n (t \nparallel u_j) \land t \nparallel w_\emptyset \}.
\end{split}
\]
For every $j \leq n-1$, $i \in \{ 0,1 \}$, pick $u_{j,i} \unrhd {u_j}^\conc i$ to be the first splitting node of $S^{n-1}_0$such that $|u_{j,i}| > |w_\emptyset|$. Finally let $A_0$ be the set of all such $u_{j,i}$'s and all $u_j$'s.
\item[\textsc{Step $l+1$}]: Assume $P_\sigma$, $w_\sigma$ and $u_{j,\sigma^\conc i}$ already constructed, for every $\sigma \in 2^{l}$, $i \in \{ 0,1 \} $. Remind that $A_l$ is the set of these $u_{j,\tau}$'s, for $\tau \in 2^{\leq l+1}$. For $i \in \{ 0,1 \}$, $\sigma \in 2^l$, find a perfect $P_{\sigma^\conc i} \subseteq P_\sigma {\downarrow} {w_{\sigma}}^\conc i$ such that, for all $s \in P_{\sigma^\conc i}$, for all $A \in \mathfrak{A}(s, S^{n-1}_l)$ we have $p(A,s)$ is not deciding.
Let
\[
\begin{split}
S^{n-1}_{l+1}  \good &  \bigcup \{ (S^{n-1}_l)_{u_{j}}, j < n \} \cup 
 \bigcup \{ p^+ \oplus s:  s \in P_\tau, \tau \in 2^{l+1} \} \cup \\
& \bigcup \{ (S^{n-1}_l)_t: t \in S^{n-1}_l \land \forall j < n (t \nparallel u_j) \land t \nparallel w_\emptyset \}.
\end{split}
\]
Then, for every $\sigma \in 2^l$, $\tau \in 2^{l+1}$, $j < n$, $i,k \in \{ 0,1 \}$, let:
\begin{itemize}
\item $w_{\sigma^\conc i} \unrhd {w_\sigma}^\conc i$ be the first splitting node in $P_{\sigma^\conc i}$;
\item $u_{j, \tau^\conc k} \unrhd {u_{j,\tau}}^\conc k$ be the first splitting node in $S^{n-1}_{l}$ such that, for all $\varsigma \in 2^{l+1}$, $|u_{j,\tau^\conc k}| > |w_{\varsigma}|$.
\end{itemize}
Finally let $A_{l+1}$ be the set of such $u_{j,{\nu}}$'s, for $\nu \in 2^{\leq l+2}$.
\end{itemize}
We keep on the construction for every $l \in \omega$ and we finally put $R = \bigcap_{\sigma \in 2^{<\omega}} P_\sigma$ and $W = \{ t: \exists s \in R (t \unlhd s)  \}$. It follows from the construction that $W$ has the required properties.

Let $S^n:= \bigcap_{l \in \omega} S^{n-1}_l $. Note that, for all $j < n$, ${u_j}^\conc 0, {u_j}^\conc 1 \in S^n \cap S^{n-1}$, and hence $S^n \leq_n S^{n-1}$.
\end{proof}

By applying iteratively Claim \ref{claim3} for every $n \in \omega$, we end up with a perfect tree $T^*:= \bigcap_{n \in \omega} S^n$ (we identify $S^0$ with the tree $T$ which we started from).
It follows from the construction that $T^*$ satisfies $(\star_2)$, and so the proof is completed.
\end{proof}
Next lemma shows that $\asacks$ satisfies the $L_f$-property, with $f(n)=4^{n}$ (\cite[Definition 7.2.1]{BJ95}). Such a property, together with axiom A, implies that $\asacks$ satisfies the Laver property, and so it does not add Cohen reals (see \cite[Lemma 7.2.2-7.2.3]{BJ95}).
\begin{lemma}
Let $A$ be a finite subset of $\omega$ and $f(n)=4^{n}$. For every $n \in \omega$, $(p_0,T) \in \asacks$ the following holds: 

if $(p_0,T) \force \dot{a} \in A$ then there exists $(p_0,T') \leq_n (p_0,T)$ and $B \subseteq A$ of size $\leq f(n)$
such that $(p_0,T') \force \dot{a} \in B$. 
\end{lemma}

\begin{proof}
Let $(p_0,T) \in \asacks$, $n \in \omega$, $A \subseteq \omega$ finite and $\dot{a}$ name for an element of $A$. We aim at finding $T' \subseteq T$ such that $(p_0,T') \leq_n (p_0,T)$ and $B$ of size $\leq 4^{n}$ such that $(p_0,T') \force \dot{a} \in B$.  First of all, pick $p \sqsupseteq p_0$ as in Remark \ref{remark:pmaximal}.

Let $q= T \restric^*l+n$, where $l:= \max\{j \in \omega: \exists t \in \splitting(p)(d(t,p)=j)\}$. We call $q^*$ a \emph{master} subtree of $q$ iff it satisfies the following property: 
\begin{itemize}
\item[$(i)$] $p \sqsubseteq q^* \subseteq q$, with $q^* \setminus p \neq \emptyset$ and $q^*$ good;
\item[$(ii)$] $\forall t \in q^*  \exists t' \unrhd t (t' \in \term(q) \cap q^*)$.
\end{itemize}
Let $\Gamma:= \{q_j: j \leq N  \}$ be the set consisting of all master subtrees of $q$. Note that $N \leq 4^{n}$; in fact, a master subtree $q^*$ is uniquely determined by what we do on the splitting nodes of $q$, and so we have four choices for each $t \in \splitting(q)$: 
\begin{enumerate}
\item $t \in \splitting(q^*)$;
\item $t \notin \splitting(q^*)$ and $t^\conc 0 \in q^*$;
\item $t \notin \splitting(q^*)$ and $t^\conc 1 \in q^*$;
\item $t \notin q^*$.
\end{enumerate}
We also remark that the upper-bound $4^n$ is not optimal, since many combinations are forbidden, by goodness.  
Then consider the following recursive construction, for $j \leq N$: 
\begin{itemize}
\item by pure decision, pick $T_0 \subseteq T {\downarrow} q_0$ and $b_0 \in \omega$ such that $(p, T_0) \force \dot{a}=b_0$. Then put $S_1 \good \bigcup \{ T_t: t \in q \setminus q_0 \} \cup T_0$.
\item for $j+1$, by pure decision, pick $T_{j+1} \subseteq S_j {\downarrow} q_{j+1}$ and $b_{j+1} \in \omega$ such that $(p, T_{j+1}) \force \dot{a}=b_{j+1}$. Then put $S_{j+1} \good \bigcup \{ (S_j)_t: t \in q \setminus q_{j+1} \} \cup T_{j+1}$.
\end{itemize}

Finally, put 
\[
T':=  T_N \text{ and } B := \{ b_j: j \leq N \}.
\]

Note that, since $q$ is good, whenever we use $\good$, we certainly do not remove any splitting node of $q$, and so $(p_0,T') \leq_n (p_0,T)$.

Given any $(q',S) \leq (p_0,T')$ there exists $j \leq k$ such that $(q_j,T' {\downarrow}q_j)$ is compatible with $(q',S)$, and therefore either $(q',S)$ does not decide $\dot{a}$ or $(q',S) \force \dot a = b_j \in B$.
Hence, we obtain $(p_0,T') \force \dot{a} \in B$ and $|B| \leq f(n)$.
\end{proof}

We conclude with an application of our amoeba-Sacks to separate regularity properties, and then with an observation concerning finite product of amoeba-Sacks. We recall some standard definitions.
\begin{enumerate} 
\item We say $X \subseteq \cantor$ to be \emph{Sacks measurable} iff 
\[ 
\forall T \in \sacks \exists T' \subseteq T, T' \in \sacks ([T'] \subseteq X \vee [T'] \cap X = \emptyset).
\]
\item Let $\Gamma$ be a certain family of sets of reals. $\Gamma(\textsc{Sacks})$ is the statement asserting that all sets in $\Gamma$ are Sacks measurable. Analogously, $\Gamma(\textsc{Baire})$ stands for all sets in $\Gamma$ have the Baire property.
\item for $z \in \cantor$, $X \subseteq \cantor$ is said to be \enfa{provable $\Delta^1_n(z)$} iff there are $\Sigma^1_n(z)$ formulae $\varphi_0$ and $\varphi_1$ such that $X=\{ x \in \cantor: \varphi_0(x) \}= \{ x \in \cantor: \neg \varphi_1(x) \}$ and $\text{ZFC} \vdash \forall x \in \cantor (\varphi_0(x) \Leftrightarrow \neg \varphi_1(x))$. The corresponding family of provable $\DDelta^1_n$ sets is denoted by \textbf{p}$\DDelta^1_n$. 
\end{enumerate}

\begin{lemma} \label{lemma:delta13}
Let $G$ be $\asacks_{\omega_1}$-generic over $\model{L}$, where $\asacks_{\omega_1}$ is the iteration of length $\omega_1$ of $\asacks$ with countable support. Then 
\[
\model{L}[G] \models \emph{\textbf{p}}\DDelta^1_3(\textsc{sacks}) \land \neg \DDelta^1_2(\textsc{Baire})
\]
\end{lemma}
\begin{proof}
Let $X \subseteq \cantor$ be defined via the $\Sigma^1_3$-formulae $\varphi_0$ and $\varphi_1$ with parameter $z \in \cantor$. Further let $\alpha < \omega_1$ such that $z \in \model{L}[G_\alpha]$, possible by properness. Let $\dot{x}$ be a name for a Sacks real over $\model{L}[G_\alpha]$. Since $X$ is provable $\Delta^1_3(z)$, it follows 
\[
\model{L}[G_\alpha] \models \text{``} \exists  T \in \sacks (T \force \varphi_0(\dot{x}) \vee T \force \varphi_1(\dot{x})) \text{"}.
\]
First assume $T \force \varphi_0(\dot{x})$, which means, for every Sacks real over $\model{L}[G_\alpha]$ through $T$, $\model{L}[G_\alpha][x] \models \varphi_0(x)$. 

Let us now argue within $\model{L}[G]$. Since $\asacks$ adds a perfect set of Sacks reals inside any perfect set from the ground model, we have a perfect tree $T' \subseteq T$ such that any $x \in [T']$ is Sacks over $\model{L}[G_\alpha]$. Hence, for every $x \in [T']$, we get $\model{L}[G_\alpha][x] \models \varphi_0(x)$, which gives $\varphi_0(x)$, by $\Sigma^1_3$-upward absoluteness. We have therefore shown that
\[
\model{L}[G] \models \exists T' \in \sacks ([T'] \subseteq  X).
\]  
Analogously, if $T \force \varphi_1(\dot{x})$ we obtain $\model{L}[G] \models \exists T' \in \sacks ([T'] \cap X= \emptyset)$.
This concludes the proof concerning $\textbf{p}\DDelta^1_3(\textsc{sacks})$. To show that $\DDelta^1_2(\textsc{Baire})$ fails it is sufficient to note that no Cohen reals are added by $\asacks_{\omega_1}$ because it satisfies the Laver property, and so $\model{L}[G] \models \neg \DDelta^1_2(\textsc{Baire})$, by well-known result proved in \cite{SJ89}.
\end{proof}
We remark that some very interesting results about $\DDelta^1_3$-measurability related to tree-forcings have been recently found by Fischer, Friedman and Khomskii in \cite{FFK14}.

\begin{remark}
Let $\asacks^*$ be the natural amoeba-Sacks adding Cohen reals. Consider the following map $\phi: \asacks^* \times \asacks^* \to \asacks^*$ such that $\langle (p_0,T_0),(p_1,T_1) \rangle$ is mapped to $(0^\conc p_0 \cup 1^\conc p_1, 0^\conc T_0 \cup 1^\conc T_1)$, where $i^\conc T:= \{ s: \exists t \in T (s=i^\conc t)  \}$, for every (possibly finite) tree in $\cantorfin$. It is straightforward to check that such $\phi$ is an isomorphism between $\asacks^* \times \asacks^*$ and $\asacks^*$ below the condition $(\{ \langle 0 \rangle, \langle 1 \rangle \}, \cantorfin)$. Hence, $\asacks^* \times \asacks^*$ completly embeds into $\asacks^*$, and so $\sacks \times \sacks$ complete embeds into $\asacks^*$ as well. In particular, we indirectly get that $\sacks \times \sacks$ is proper. 

Finally, note that such an argument holds for any finite product $(\asacks^*)^n$. In fact, fixed $n \in \omega$, let $t_0,t_1, \dots t_{n-1}$ be a list of  $n$-many sequences of $\cantorfin$ which are pairwise incompatible. Then let $\phi: (\asacks^*)^n \to \asacks^*$ be such that 
\[
\phi(\langle (p_j,T_j): j < n \rangle)= \big( \bigcup_{j<n} t_j^\conc p_j, \bigcup_{j<n} t_j^\conc T_j \big).
\]
As above, $\phi$ is an isomorphism between  $(\asacks^*)^n$ and $\asacks^*$ below the condition $(p^*,\cantorfin)$, with $p^*$ the finite tree generated by $t_0, \dots t_{n-1}$, i.e., the set of initial segments of sequences in $\bigcup_{j < n} t_j$.

Note that this argument is no longer true for our amoeba-Sacks $\asacks$ analyzed in this paper. In fact, it is easy to see that $\asacks$ adds a dominating real. Now consider the product $\asacks \times \asacks$ and let $d_0, d_1$ be a pair of mutually dominating reals added by $\asacks \times \asacks$. Define the real $c$ as follows: $c(n)=0$ iff $d_0(n) \leq d_1(n)$, $c(n)=1$ otherwise. Such $c$ is obviously Cohen, since we can freely make either $d_0(n) > d_1(n)$ or $d_0(n) < d_1(n)$, and hence $\asacks \times \asacks$ does not completely embed into $\asacks$. 
\end{remark}


\section{Concluding remarks}
Many difficulties come out when trying to remove the pathology of Remark \ref{remark:silver-cohen} about amoeba-Silver, as we did for amoeba-Sacks.
A first idea to remove Cohen reals could be to oblige the Silver tree $T$ of the pair $(p,T)$ to have always an even number of 0s between two subsequent splitting nodes. 
Nevertheless, even if this modification formally removes the Cohen real defined as in Remark \ref{remark:silver-cohen}, it cannot suppress \emph{any} Cohen real; in fact, putting $\Gamma_n = \{ j \in \omega: \levnumber_{T_G}(n+1) + 2 < j \leq \levnumber_{T_G}(n+2)+1 \land T_G(j)=0 \}$, one can similarly define a Cohen real by putting $c(n)=0$ iff $|\Gamma_n|=0$ modulo $3$ (and $c(n)=1$ otherwise). More generally, one could fix a new condition saying that the number of 0s between two splitting levels of $T$ has to be a multiple of a given sequence of natural numbers $n_0, \dots n_k$; in any case, this will not settle the problem, since one could define a \emph{new} Cohen real such that $c(n)=0$ iff $|\Gamma_n|=0$ modulo $n_0 \cdot n_1 \cdot \dots \cdot n_k+1$. Furthermore, if we look at the construction of the amoeba-Sacks, one can realize that it does not work for the amoeba-Silver; in fact, we cannot freely remove splitting nodes as in claim \ref{claim1}, since we have to respect the uniformity of the Silver tree.

As we said in the introduction, the notion of ``amoeba'' is meant as a ``forcing adding a \emph{large} set of \emph{generic} reals'', where the words ``large'' and ``generic'' depend on the forcing we are dealing with. In the examples we have mentioned and studied in the previous sections, ``large'' and ``generic'' were connected to the same forcing notion; in fact, we have considered an amoeba-Sacks adding a \emph{Sacks} tree of \emph{Sacks} branches and an amoeba-Silver adding a \emph{Silver} tree of \emph{Silver} branches. Furthermore, the usual amoeba for measure and category add a \emph{measure one }set of \emph{random} reals and a \emph{comeager} set of \emph{Cohen} reals, respectively. What can also be done is to consider amoeba for which the notion of ``large'' and the one of ``generic'' are not necessarily connected. As a simple example, one can consider the Cohen forcing, viewed as a forcing adding a perfect tree of Cohen branches. Or otherwise, one could pick the forcing $\poset{RT}$ consisting of pairs $(p,T)$, where $T \subseteq \cantor$ is a perfect tree with positive measure and $p \subset T$ is a finite subtree. It is clear that such a forcing adds a perfect tree of random reals. These two examples are particular cases of a more general definition.
\begin{definition}
Let $\poset{P}_0$ and $\poset{P}_1$ be tree-forcing notions. We say that a forcing $Q$ is a $(\poset{P}_0,\poset{P}_1)$-amoeba iff for every $p \in \poset{P}_1 \cap \model{N}$ there is $p' \in \poset{P}_0 \cap \model{N}^Q$ such that $p' \subseteq p$ and 
\[
\model{M} \models \text{`` every branch $x \in [p']$ is $\poset{P}_1$-generic over $\model{N}$ ''},
\]
where $\model{M}$ is \emph{any} model of ZFC containing as a subset the extension of $\model{N}$ via $Q$. 
\end{definition}
Hence, the forcing $\poset{RT}$ mentioned above is an $(\sacks,\random)$-amoeba, while Cohen forcing can be seen as an $(\sacks, \cohen)$-amoeba. Note that this general version of amoeba can be useful to obtain some results regarding regularity properties, such as that in lemma \ref{lemma:delta13}. In fact, a similar proof shows that an $\omega_1$-iteration of $\poset{RT}$ provides a model for $\textbf{p}\DDelta^1_3(\textsc{sacks})$ as well. However, the two iterations are different. In fact, $\poset{RT}$ adds Cohen reals but not dominating reals. The latter is proven in \cite[lemma 3.2.24, 6.5.10 and theorem 6.5.11]{BJ95}, whereas the former can be shown as follows: pick an interval partition $\{I_n: n \in \omega \}$ of $\omega$ such that, for all but finitely many $n$, any perfect tree of positive measure has at least one splitting node of length occurring in $I_n$. Then define the real $x \in \cantor$ such that $x(n)=$ the parity of splitting nodes of $T_G$ occurring in $I_n$, where $T_G$ is the $\poset{RT}$-generic tree given by $\bigcup\{ p: \exists T((p,T) \in G ) \}$. It is straightforward to check that $x$ is a Cohen real. Hence if we pick $\poset{RT}_{\omega_1}$ to be the $\omega_1$-iteration of $\poset{RT}$ with finite support (this to make sure that no dominating reals are added by the iteration), we obtain a model for $\textbf{p}\DDelta^1_3(\textsc{sacks}) \land \neg \DDelta^1_2(\textsc{Laver}) \land \DDelta^1_2(\textsc{Baire})$, where Laver measurability is defined analogously as Sacks measurability, and we use \cite[Theorem 4.1]{BL99} to obtain $\neg \DDelta^1_2(\textsc{Laver})$. Hence, such a model is different from the one presented in lemma \ref{lemma:delta13} satisfying $\textbf{p}\DDelta^1_3(\textsc{sacks}) \land \DDelta^1_2(\textsc{Laver}) \land \neg \DDelta^1_2(\textsc{Baire})$. 

These observations, together with Remark \ref{remark:silver-cohen}, give rise to the following interesting questions:
\begin{itemize}
\item[(Q1)] Can one define an amoeba-Sacks not adding either Cohen or dominating reals?
\item[(Q2)] Can one define an amoeba-Silver not adding Cohen reals? And/or not adding either Cohen or dominating reals?
\item[(Q3)] Does ``adding a perfect tree of random branches'' imply ``adding Cohen reals''?
\end{itemize}

\end{document}